\documentclass[12pt,reqno]{amsart}
\usepackage[utf8]{inputenc}
\usepackage[T1]{fontenc}
\usepackage{amsmath,amssymb,amsthm,amsfonts,mathrsfs}
\usepackage{hyperref}
\usepackage{geometry}
\usepackage{tikz}
\usepackage{pgfplots}
\usepackage{caption}
\usepackage{mathtools}
\pgfplotsset{compat=1.17}
\usetikzlibrary{calc}
\usepackage{pstricks}
\usepackage{pst-plot}
\usepackage{pst-node}
\usepackage{bm}
\usepackage{enumitem}
\usepackage{array}
\usepackage{makecell}
\usepackage{color}
\usepackage[dvipsnames]{xcolor}

\newcommand{\cemph}[1]{{\color{blue!70!black}\emph{#1}}}

\allowdisplaybreaks[4]

\numberwithin{equation}{section}
\newtheorem{theorem}{Theorem}[section]
\newtheorem{lemma}[theorem]{Lemma}

\newtheorem{conjecture}[theorem]{Conjecture}

\newtheorem{example}[]{Example}

\newtheorem{remark}{Remark}[section]

\renewcommand{\S}{\mathfrak{S}}

\def\ides{\mathrm{ides}}
\def\pix{\mathrm{pix}}

\def\Asct{\mathrm{Asct}}

\def\Des{\mathrm{Des}}

\def\des{\mathrm{des}}

\def\fix{\mathrm{fix}}

\usepackage{color}
\usepackage[dvipsnames]{xcolor}

\usepackage[backend=biber,backref=true,style=alphabetic,sorting=nyt,giveninits=true,doi=false,url=false,isbn=false,maxbibnames=99,
  minbibnames=99]{biblatex}
\renewbibmacro*{in:}{}

\title[Fixed-Pixed points equidistribution]{A Refinement of the Fixed--Pixed Points Equidistribution on restricted Permutations}
\author{Yao Dong and Chao Xu}
\address{College of Mathematics and Statistics, Northwest Normal University, Lanzhou
730070, P.R. China}
\email{dongy@nwnu.edu.cn}
\address{Universite Lyon 1,
UMR 5208 du CNRS, Institut Camille Jordan\\
F-69622, Villeurbanne Cedex, France}
\email{xu@math.univ-lyon1.fr}

\date{\today}
\begin{document}
\begin{abstract}
Motivated by a recent conjecture of Bsila, Cox, Hugo, Styron and Zhuang concerning fixed points and pixed points on pattern-avoiding permutations, we prove a bivariate refinement involving descent statistics.  Given a set of permutations $\Pi$, let $\S_n(\Pi)$ denote the set of permutations in the symmetric group $\S_n$ that avoid every element of $\Pi$ in the sense of pattern avoidance. For each set $\Pi$ appearing in their conjecture, we show that the pairs of statistics $(\des,\fix)$ and $(\ides,\pix)$
are equidistributed over $\S_n(\Pi)$. Our proof is based on explicit ordinary generating functions for the corresponding pattern-avoiding classes.  

\end{abstract}

\maketitle

\section{Introduction}
Let $\S_n$ denote the symmetric group of permutations of the set $[n]=\{1,2,\ldots,n\}$. 
We call $i\in[n]$ a \cemph{fixed point} of $\sigma\in\S_n$ if $\sigma(i)=i$. Let $\fix(\sigma)$ denote the number of fixed points of $\sigma$. A permutation $\sigma\in\mathfrak S_n$ is called a \cemph{derangement} if $\fix(\sigma)=0$.  Researches on fixed points and derangements can refer to References \cite{BCCHZ05,DJW93,DZ01,HX09,L16,S12}.
We call $i$ an \cemph{ascent} of $\sigma$ if $\sigma(i)<\sigma(i+1)$ for $i\in[n-1]$ or $i=n$.
In the 1980s, D\'esarm\'enien introduced another family of permutations, now called \cemph{desarrangements}, and proved that they are equinumerous with derangements \cite{Des84}. A desarrangement is  a permutation whose first ascent is even. Note that a decreasing permutation $nn-1\cdots1$ is a desarrangement if and only if $n$ is even. The empty word is also regarded as a desarrangement.
\begin{example} 
For \(1\le n\le 4\), the desarrangements in \(\mathfrak S_n\) are, respectively,
\[
\varnothing,\quad \{21\},\quad \{213,312\},
\]
and
\[
\{2134,2143,3124,3142,3241,4123,4132,4231,4321\}.
\]
\end{example}

The intrinsic relationship between derangements and desarrangements has been further explored on the basis of the definition of pixed points. As a  combinatorial statistic on permutations, pixed points were originally proposed and thoroughly studied by Foata and Han in the setting of hyperoctahedral groups. 
Foata and Han~\cite{FH08} introduced the pixed factorization of a permutation, that is, every permutation $\sigma\in\S_n$ admits a unique factorization $\sigma=\iota\delta$, where $\iota$ is an increasing prefix and $\delta$ is a desarrangement. The letters of $\iota$ are the \cemph{pixed points} of $\sigma$, and we denote by $\pix(\sigma)$ the number of pixed points of $\sigma$.  In particular, the identity permutation  has $n$ pixed points. Under this terminology, derangements are the permutations with no fixed points, while desarrangements are the permutations with no pixed points. 
\begin{example}
Let $\sigma=135764829\in\mathfrak S_9$. Then the pixed factorization of $\sigma$ is
\[
\sigma=\underbrace{1357}_{\iota}\underbrace{64829}_{\delta}.
\]
Here $\iota=1357$ is the longest increasing prefix of $\sigma$ and the first ascent of $\delta=64829$ occurs at position $2$. Thus $\delta$ is a desarrangement and $\pix(\sigma)=4.$
\end{example}

Pattern avoidance is a central topic in the modern study of permutations.  Given two permutations
$\sigma\in\mathfrak S_n$ and $\tau\in\mathfrak S_k$, we say that $\sigma$ contains the pattern
$\tau$ if there exist indices
$1\leq i_1<i_2<\cdots<i_k\leq n$
such that the subsequence
$\sigma(i_1)\sigma(i_2)\cdots\sigma(i_k)$
has the same relative order as $\tau$.  Otherwise, we say that $\sigma$ \cemph{avoids} $\tau$. For example, the permutation $\sigma=31524$ contains the pattern $132$, since the subsequence
$354$ has the same relative order as $132$. On the other hand, $\sigma=31524$ avoids pattern $321$, since it has no decreasing subsequence of length $3$. More generally, for a set of patterns $\Pi$, we write
\begin{equation}
  \mathfrak S_n(\Pi)
=
\{\sigma\in\mathfrak S_n:\sigma\text{ avoids every pattern in }\Pi\}.  \end{equation}
When the patterns in \(\Pi\) are listed explicitly, we often omit the enclosing braces. For example,
$\mathfrak S_n(123,132)$ stands for
$\mathfrak S_n(\{123,132\})$.

The study of pattern-avoiding permutations goes back to the work of Knuth on stack-sortable permutations, where the permutations sortable by a single stack were shown to be precisely the $231$-avoiding permutations \cite{Knu98}.  These permutations are counted by the Catalan numbers, and this observation initiated a rich interaction between permutation patterns and classical Catalan objects.  Later, Simion and Schmidt \cite{SS85} carried out a systematic study of permutations avoiding patterns of length three, proving in particular that all single patterns of length three are Catalan-enumerated. Over the past few decades, pattern avoidance has become a broad framework for studying refined enumeration, bijective constructions, permutation statistics, and structural decompositions of permutation classes; see \cite{Kit11,MR02} for an overview.

Recently, Bsila, Cox, Hugo, Styron and Zhuang revisited desarrangements from the point of view of permutation statistics and pattern avoidance \cite{BCCHZ05}. They obtained several generating functions for desarrangements with respect to some permutation statistics, including descents, peaks, valleys, double ascents and double descents, and they also gave a complete enumeration of desarrangements avoiding prescribed sets of patterns of length three. Their work led to new interpretations of several classical integer sequences, including Catalan, Fine, Jacobsthal and Fibonacci numbers, in terms of pattern-avoiding desarrangements.

Let 
\begin{equation}\label{family:P}
    \mathcal P=
\left\{
\begin{array}{ccc}
\{132,312\}, & \{132,321\}, & \{213,231\},\\[2mm]
\{123,132,312\}, & \{123,213,231\}, & \{123,312,321\},\\[2mm]
\{132,312,321\}, & \{213,231,312\}, & \{213,231,321\}
\end{array}
\right\}.
\end{equation}
At the end of their paper, Bsila et al.~\cite{BCCHZ05} posed the following conjecture~\footnote{After completing our work, we learned that Zhuang and his students had independently obtained a bijective proof of the original conjecture~\cite{Zhuangnote}.}.
\begin{conjecture}
For all $n\geq 0$ and for every $\Pi\in\mathcal{P}$, the permutation statistics $\operatorname{fix}$ and $\operatorname{pix}$ are equidistributed over $\mathfrak S_n(\Pi)$. Equivalently,
\begin{equation}
    \sum_{\pi\in\mathfrak S_n(\Pi)}
x^{\operatorname{fix}(\sigma)}
=
\sum_{\pi\in\mathfrak S_n(\Pi)}
x^{\operatorname{pix}(\sigma)}.
\end{equation}
\end{conjecture}

The purpose of this paper is to prove a refinement of this conjecture. Instead of considering only the one-variable distributions of $\operatorname{fix}$ and $\operatorname{pix}$, we include descent statistics. We call $i\in[n-1]$ a \cemph{descent} of $\sigma\in\S_n$ if $\sigma(i)>\sigma(i+1)$. We denote by  $\Des(\sigma)$ the set of descents of $\sigma$, and denote its cardinality  by $\des(\sigma)=|\Des(\sigma)|$. We call \(i\in[n-1]\) an \cemph{inverse descent} of \(\sigma\in\mathfrak S_n\)
if \(\sigma^{-1}(i)>\sigma^{-1}(i+1)\), and we denote by
\(\operatorname{ides}(\sigma)\) the number of inverse descents of \(\sigma\).
Equivalently, \(i\) is an inverse descent of \(\sigma\) if \(i+1\) appears
to the left of \(i\) in \(\sigma\).
 Note that $\ides(\sigma)$ is simply the number of $i\in[n-1]$  such that  $i+1$ appears somewhere to the left of $i$ in $\sigma$.

\begin{example}Let $\sigma=314265\in\S_6$.
   The descents of $\sigma$  are $1,3$ and $5$, since
\(3>1\), \(4>2\), and \(6>5\). Hence \(\des(\sigma)=3\). Moreover,
\(\sigma^{-1}=241365\), whose descents are \(2\) and \(5\). Equivalently,
the inverse descents of \(\sigma\) are obtained by comparing the positions
of \(i\) and \(i+1\): here \(3\) appears to the left of \(2\), and \(6\)
appears to the left of \(5\). Therefore, the inverse descents of \(\sigma\)
are \(2\) and \(5\), and
$\ides(\sigma)=2.$
\end{example}

For each pattern set $\Pi\in\mathcal{P}$, we define the following two enumerative polynomials
\begin{equation}\label{gen-F}
    F_\Pi(t;x,y)
:=
\sum_{n\geq 0}
\left(
\sum_{\pi\in\mathfrak S_n(\Pi)}
x^{\operatorname{des}(\pi)}y^{\operatorname{fix}(\pi)}
\right)t^n
\end{equation}
and
\begin{equation}\label{gen-P}
    P_\Pi(t;x,y)
:=
\sum_{n\geq 0}
\left(
\sum_{\pi\in\mathfrak S_n(\Pi)}
x^{\operatorname{ides}(\pi)}y^{\operatorname{pix}(\pi)}
\right)t^n.
\end{equation}
The following theorem is our main result.
\begin{theorem}\label{thm:main}
For every $\Pi\in\mathcal P$, we have
\begin{equation}
   F_{\Pi}(t;x,y)=P_{\Pi}(t;x,y). 
\end{equation}
Equivalently, for $n\ge 0$, the two pairs  $(\des,\fix)$ and $(\ides,\pix)$ are equidistributed over $\mathfrak S_n(\Pi)$.
\end{theorem}
\begin{remark}
For a pattern set $\Pi\in\mathcal{P}$, let
\begin{equation}
    d_n(\Pi):=
|\{\sigma\in\mathfrak S_n(\Pi):\operatorname{fix}(\sigma)=0\}|
\text{ and }
\widetilde d_n(\Pi):=
|\{\sigma\in\mathfrak S_n(\Pi):\operatorname{pix}(\sigma)=0\}|.
\end{equation}
Thus \(d_n(\Pi)\) (resp. \(\widetilde d_n(\Pi)\)) is the number of  derangements (resp. desarrangements) in \(\mathfrak S_n(\Pi)\). 
It follows from setting \(x=1\) and $y=0$ in Theorem~\ref{thm:main} that
\begin{equation}
    d_n(\Pi)=\widetilde d_n(\Pi),
\end{equation}
for all \(n\geq 0\) and all \(\Pi\in\mathcal P\). 
The numbers \(d_n(\Pi)\) can be obtained from the study of enumerating the fixed point refinement of  pattern-avoiding permutations by
Robertson--Saracino--Zeilberger~\cite{RSZ02} and
Mansour--Robertson~\cite{MR02}.
Theorem~\ref{thm:main} refines the nine pattern classes belonging to $\mathcal{P}$ in Theorem~3.28 of Bsila et al.~\cite{BCCHZ05}. Note that the additional case
\(\Pi=\{132\}\) appearing in their theorem is not covered by this
refinement, since it does not satisfy the full equidistribution of $\fix$ and $\pix$.
\end{remark}

We shall prove Theorem~\ref{thm:main} by generating function techniques.
The two-pattern and three-pattern cases are treated in Sections~\ref{sec:2} and~\ref{sec:3}, respectively.

\section{Proof of Theorem~\ref{thm:main}}
Before proving Theorem~\ref{thm:main}, we recall two basic operators on permutations. For any $\sigma\in\S_n$, its \cemph{reversal} $\sigma^r\in\S_n$ is given by $\sigma^r(i)=\sigma(n+1-i)$; its \cemph{complement} $\sigma^c\in\S_n$ is given by  $\sigma^c(i)=n+1-\sigma(i)$. The \cemph{reversal-complement} of $\sigma$  is defined by 
\begin{equation}
    \sigma^{rc}:=(\sigma^r)^c=(\sigma^c)^r.
\end{equation}
Equivalently,
\begin{equation}\label{com-re-operator}
    \sigma^{rc}(i)=n+1-\sigma(n+1-i)\,\text{ for } 1\le i\le n.
\end{equation}
In what follows, we shall use the following elementary fact repeatedly.
\begin{lemma}\label{re-com}
  The reverse-complement map $\sigma\mapsto\sigma^{rc}$ preserves the statistics $\des$ and $\fix$. For any $\sigma\in \S_n$, we have 
 \begin{equation}
    \des(\sigma^{rc})=\des(\sigma)\text{ and } \fix(\sigma^{rc})=\fix(\sigma).
\end{equation}
\end{lemma}
\begin{proof}
For \(1\le i\le n-1\), by the definition of \(\sigma^{rc}\), we have
\[
\begin{aligned}
i\in\Des(\sigma^{rc})
&\Longleftrightarrow \sigma^{rc}(i)>\sigma^{rc}(i+1)  \\
&\Longleftrightarrow n+1-\sigma(n+1-i)>
n+1-\sigma(n-i) \\
&\Longleftrightarrow \sigma(n+1-i)<\sigma(n-i) \\
&\Longleftrightarrow n-i\in\Des(\sigma),
\end{aligned}
\]
and for $1\le i\le n$,
\[
\begin{aligned}
\sigma^{rc}(i)=i
&\Longleftrightarrow n+1-\sigma(n+1-i)=i  \\
&\Longleftrightarrow \sigma(n+1-i)=n+1-i.
\end{aligned}
\]
 It follows that
\(\sigma\mapsto\sigma^{rc}\) preserves both \(\des\) and \(\fix\).

\end{proof}

\subsection{Double-avoidance classes in \texorpdfstring{$\mathcal{P}$}{P}}\label{sec:2}

In this section, we prove Theorem~\ref{thm:main} for three double-avoiding classes $\Pi$, namely $\{132,312\}$,
$\{213,231\}$ and $ \{132,321\}$, by computing their ordinary generating functions $F_{\Pi}(t;x,y)$ and $P_{\Pi}(t;x,y)$, respectively.

\subsubsection{The class $\S_n(132,312)$} 
Given a permutation $\sigma\in\S_n$, we call $i\in\{2,3,\ldots,n\}$ an \cemph{ascent top position} (resp. a \cemph{descent bottom position}) if $\sigma(i-1)<\sigma(i)$ (resp. $\sigma(i-1)>\sigma(i)$). Let $\mathrm{Asct}(\sigma)$ be the set of all ascent tops of $\sigma$. 
A well-established bijection $\xi$ between the restricted permutation class $\S_n(132,312)$ and the power set $2^{\{2,3,\ldots,n\}}$ originally introduced by Simion and Schmidt \cite{SS85}, is defined by 
\begin{equation}\label{bijection-xi}
    \sigma\longmapsto \xi(\sigma):= \Asct(\sigma).
\end{equation}
Inversely, for any given subset $A=\{a_1<a_2<\cdots<a_m\}\in2^{\{2,3,\ldots,n\}}$,  the corresponding permutation $\sigma$ is uniquely constructed by  setting
\begin{equation}\label{bi-xi}
   \sigma(a_i)=n-m+i\,\text{ for }1\le i\le m, 
\end{equation}
while the remaining positions are filled, from left to right, by
\(n-m,n-m-1,\ldots,1\) in strictly decreasing order. 
The following lemma shows how this bijection provides a natural way to evaluate the permutation statistics $\des$ and $\ides$ on $\S_n(132,312)$.

\begin{lemma}\label{lemma-xi}
Let $\sigma\in\S_{n}(132,312)$, and set  
\begin{equation}\label{def:B}
    B=[n]\setminus \Asct(\sigma).
\end{equation} Then we have
\begin{equation}\label{equ:des-ides}
    \des(\sigma)=\ides(\sigma)=|B|-1.
\end{equation}
\end{lemma}
\begin{proof} 
We first compute $\operatorname{des}(\sigma)$. Since every index $i\in\{2,\dots,n\}$ is either an ascent top position or a descent bottom position of $\sigma$,  we see that indices in $\{2,\dots,n\}\setminus \Asct(\sigma)$ are exactly the descent bottoms of $\sigma$. Then, by \eqref{def:B}, $\des(\sigma)=|B|-1$.

It remains to compute $\operatorname{ides}(\sigma)$. Set $m=|\Asct(\sigma)|$. If $m=0$, then $\sigma=n(n-1)\cdots1$, and hence $\ides(\sigma)=n-1=|B|-1$. Thus, assume $m\ge 1$. By \eqref{bi-xi}, the permutation $\sigma$ is constructed by placing the $m$ largest values $\{n-m+1, \dots, n\}$ in strictly increasing order at positions in $\Asct(\sigma)$, and the remaining $n-m$ values $\{1, \dots, n-m\}$ in strictly decreasing order at positions in $B$. An inverse descent occurs for a value $v$ if and only if $v+1$ appears to the left of $v$ in $\sigma$. One can check that
\begin{itemize}
    \item For $v \in \{n-m+1, \dots, n-1\}$: These values are placed in increasing order, so $v+1$ is always to the right of $v$. Hence, there are no inverse descents.
    \item For $v \in \{1, \dots, n-m-1\}$: These values are placed in strictly decreasing order from left to right. Thus, the larger value $v+1$ is always placed before (to the left of) $v$. Hence, there are exactly $n - m - 1$ inverse descents.
    \item For $v = n-m$:  Since $1 \notin \Asct(\sigma)$, position $1$ belongs to $B$, giving $\sigma(1) = n-m$. The value $v+1 = n-m+1$ is placed at the first ascent position $a_1 \ge 2$. Thus, $v$ is to the left of $v+1$, which means that $v$ is not an inverse descent.
\end{itemize}
It follows from the above argument that $\ides(\sigma)=|B|-1$.
\end{proof}
\begin{example}
Let $\sigma=6\,7\,5\,4\,8\,9\,3\,10\,2\,1\in\S_{10}$. On one hand, we have $\Asct(\sigma)=\{2,5,6,8\}$ and $B=[10]\setminus\Asct(\sigma)=\{1,3,4,7,9,10\}$. On the other hand, we have $\des(\sigma)=|\{2,3,6,8,9\}|$ and $\ides(\sigma)=|\{1,2,3,4,5\}|$. Thus $\des(\sigma)=\ides(\sigma)=|B|-1.$
\end{example}

\begin{theorem}\label{prop:132312}
Let $T_1=\{132,312\}$. We have
\begin{equation}\label{gen-F-132312}
    F_{T_1}(t;x,y)
=1+\frac{yt}{1-yt}
+\frac{x t^2(1+xyt)(1-xt)}{(1-yt)(1-x^2t^2)(1-(1+x)t)}.
\end{equation}
\end{theorem}


\begin{proof}

Suppose that  $B=\{b_1<b_2<\cdots<b_N\}$.
Under the reverse construction of the Simion--Schmidt bijection $\xi$ defined by \eqref{bi-xi}, we have $\sigma(b_j)=N+1-j \text{ for }
 1\leq j\leq N$, i.e., the entries in the
positions belonging to \(B\) are filled, from left to right, by
$N,N-1,\ldots,1.$
It follows that \(b_j\) is a fixed point if and only if
\begin{equation}
    b_j=N+1-j,
\end{equation}
or equivalently,
\begin{equation}
    b_j+j=N+1.
\end{equation}
Since the sequence \(b_j+j\) is strictly increasing, this can happen for at
most one index \(j\). Define
\begin{equation}\label{equ:epsilon-B}
    \epsilon(B)=
\begin{cases}
1,&\text{if there exists }j\text{ such that }b_j+j=N+1;\\[4pt]
0,&\text{otherwise.}
\end{cases}
\end{equation}

Since \(b_N\) is the maximum element of the set \(B\), it follows that all positions \(b_N + 1, b_N + 2, \ldots, n\) are contained in \(\Asct(\sigma)\).  It is easy to see that by the reverse construction of $\xi$, these positions are forced to be fixed points. It follows that
\begin{equation}\label{equ:fix-point}
    \fix(\sigma)=n-b_N+\epsilon(B).
\end{equation}

By combining the bijection $\xi$ with Lemma~\ref{lemma-xi} and~\eqref{equ:fix-point}, the polynomials $F_{\Pi}(t;x,y)$ (for $\Pi=T_1=\{132,312\}$) defined in~\eqref{gen-F} can be reformulated in the following manner,
\begin{equation}
   F_{T_1}(t;x,y)=1+\sum_{n\ge 1}\sum_{\substack{B\subseteq[n]\\1\in B}}x^{|B|-1}y^{n-\max B+\epsilon(B)}t^n, 
\end{equation}
where $\max B$ denotes the maximum element of set $B$. By first fixing $B$ and summing over all $n\ge \max B$, we obtain
\begin{align}\label{F-set-version}
F_{T_1}(t;x,y)
&=
1+
\sum_{\substack{B\subseteq \mathbb{P}\\1\in B}}
x^{|B|-1}y^{\epsilon(B)}
\sum_{n\geq \max B}y^{n-\max B}t^n\nonumber  \\
&=
1+
\frac{1}{1-yt}
\sum_{\substack{B\subseteq \mathbb{P}\\1\in B}} x^{|B|-1}y^{\epsilon(B)}t^{\max B}.
\end{align}
Define the enumerative polynomials
\begin{equation}\label{H-set}
    H(t;x,y):=
\sum_{\substack{B\subseteq \mathbb{P}\\1\in B}} x^{|B|-1}y^{\epsilon(B)}t^{\max B}.
\end{equation}
We now compute \(H(t;x,y)\). First ignore the statistic \(\epsilon(B)\), and set
\begin{equation}\label{equ:H_0}
    H_0(t;x):=\sum_{\substack{B\subseteq \mathbb{P}\\1\in B}} x^{|B|-1}t^{\max B}.
\end{equation}
If \(B=\{1\}\), the contribution is \(t\). If \(\max B=L\geq 2\), then
\(1\) and \(L\) are forced to be in \(B\), while the elements of
\(\{2,3,\ldots,L-1\}\) may be chosen freely. Hence, 
\begin{equation}\label{gen-fun-H_0}
    H_0(t;x)=
t+\sum_{L\geq 2}x(1+x)^{L-2}t^L.
\end{equation}
Next,  we compute the contribution of those sets $B$ satisfying $\epsilon(B)=1$. Let
\begin{equation}\label{def:H_1}
    H_1(t;x):=\sum_{\substack{B\subseteq \mathbb{P}\\1\in B,\,\epsilon(B)=1}}x^{|B|-1}t^{\max B}.
\end{equation}
The set \(B=\{1\}\) satisfies $\epsilon(B)=1$ and contributes \(t\). Now suppose that $B=\{b_1<b_2<\cdots<b_N\}$ has \(N\geq 2\) elements and satisfies \(\epsilon(B)=1\). Then there exist a unique  index \(j\)  such that
\begin{equation}\label{relation:b_j-N}
    b_j+j=N+1.
\end{equation}
Set \(p=b_j\). Then \(N=p+j-1\). Since \(1=b_1\) and \(b_j=p\), the
elements of \(B\) before \(p\) consist of \(1\), together with \(j-2\)
elements chosen from \(\{2,3,\ldots,p-1\}\). Thus there are
\[
\binom{p-2}{j-2}
\]
choices for the part of \(B\) before \(p\).

On the other hand, the number of elements of $B$ after $p$ is \(N - j = b_j - 1\). If $L=\max B$, then one of these elements is $L$, and the remaining $p-2$ elements are chosen from $\{p+1,p+2,\ldots,L-1\}.$
Thus, for fixed \(p\) and \(L\), the number of choices after \(p\) is
$$\binom{L-p-1}{p-2}.$$
Summing over all possible \(L\), we obtain
\begin{equation}
    \sum_{L\ge 2p-1}\binom{L-p-1}{p-2}t^L
=
\frac{t^{2p-1}}{(1-t)^{p-1}}.
\end{equation}
Moreover, since \(N=p+j-1\), the exponent of \(x\) is
$|B|-1=N-1=p+j-2.$ Therefore
\begin{equation}
    H_1(t;x)
=
t+
\sum_{p\ge 2}\sum_{j=2}^{p}
\binom{p-2}{j-2}
x^{p+j-2}
\frac{t^{2p-1}}{(1-t)^{p-1}}.
\end{equation}
Setting \(k=p-2\) and \(i=j-2\), we get

\begin{equation}
H_1(t;x)-t=
\sum_{k\ge 0}
\frac{x^{k+2}t^{2k+3}}{(1-t)^{k+1}}
\sum_{i=0}^{k}\binom{k}{i}x^i
\end{equation}
Combining the binomial theorem  and the geometric series identity, we derive 
\begin{equation}\label{gen-fun-H_1}
    H_1(t;x)=
t+\frac{x^2t^3}{1-t-x(1+x)t^2}.
\end{equation}
Since \(y\) only records whether the possible extra fixed point inside
\(B\) exists, we have
\begin{equation}\label{relation-H-H_0-H_1}
    H(t;x,y)=H_0(t;x)+(y-1)H_1(t;x).
\end{equation}
Substituting \eqref{gen-fun-H_0} and \eqref{gen-fun-H_1} into \eqref{relation-H-H_0-H_1}, and using
\[
1-t-x(1+x)t^2=(1+xt)(1-(1+x)t),
\]
we obtain
\begin{equation}
    H(t;x,y)
=
yt+
\frac{xt^2(1+xyt)(1-xt)}
{(1-x^2t^2)(1-(1+x)t)}.
\end{equation}
Combining this with \eqref{F-set-version} and \eqref{H-set}, we obtain \eqref{gen-F-132312}.
\end{proof}

\begin{theorem}
Let $T_1=\{132,312\}$.  We have
\begin{equation}\label{gen-P-132312}
    P_{T_1}(t;x,y)
=
1+\frac{yt}{1-yt}
+
\frac{x t^2(1+xyt)(1-xt)}
{(1-yt)(1-x^2t^2)(1-(1+x)t)}.
\end{equation}
\end{theorem}
\begin{proof}
We first consider the case $\Asct(\sigma)=\{2,3,\ldots,n\}$. Then $\sigma=12\cdots n$ and $B=[n]\setminus \Asct(\sigma)=\{1\}$. Hence, by the pixed factorization, we have $\pix(\sigma)=n$, and by Lemma~\ref{lemma-xi}, we have $\ides(\sigma)=0$. Therefore, this case and the empty permutation contribute
\begin{equation}\label{gen-trival}
    1+\sum_{n\geq 1} y^n t^n
=1+
\frac{yt}{1-yt}.
\end{equation}
For the remaining part, assume that \(|B|\geq 2\) and 
$B=\{b_1<b_2<\cdots<b_N\}$ with $b_1=1$.
Since \(B=[n]\setminus \Asct(\sigma)\), the positions
$2,3,\ldots,b_2-1$
belong to \(\Asct(\sigma)\). Hence the maximal
increasing prefix of \(\sigma\) has length $b_2-1$. Now let \(r\geq 1\) be the length of the maximal consecutive segment of elements
of \(B\) beginning at \(b_2\), i.e., 
\begin{equation}
    b_2,b_2+1,\ldots,b_2+r-1\in B,
\end{equation}
and either \(b_2+r\notin B\) or \(b_2+r>n\). 
Define 
\begin{equation}\label{eta-B}
     \eta(B)=
\begin{cases}
0,&\text{if $r$ is even;}\\[4pt]
1,&\text{if $r$ is odd.}
\end{cases}
\end{equation}
One can check that the number of  pixed points of $\sigma$ is decided by the parity of $r$,
\begin{equation}\label{equ:pix-B}
   \pix(\sigma)=b_2-1-\eta(B), 
\end{equation}
or equivalently,
\begin{equation}
    \pix(\sigma)=
\begin{cases}
b_2-1, & \text{if } r \text{ is even};\\[4pt]
b_2-2, & \text{if } r \text{ is odd}.
\end{cases}
\end{equation}
By combining the bijection $\xi$ with Lemma~\ref{lemma-xi}, \eqref{equ:pix-B} and \eqref{gen-trival}, the polynomials $P_{\Pi}(t;x,y)$ (for $\Pi=T_1=\{132,312\}$), defined in \eqref{gen-P}, can be reformulated as follows:
\begin{equation}
    P_{T_1}(t;x,y)=1+\frac{yt}{1-yt}+\sum_{n\ge 2}\sum_{\substack{B\subseteq[n]\\1\in B,\, |B|\ge 2}}x^{|B|-1}y^{b_2-1-\eta(B)}t^n.
\end{equation}
By first fixing $B$ and summing over all $n\ge \max B$, we obtain
\begin{align}\label{rewrite-P}
    P_{T_1}(t;x,y)&=1+\frac{yt}{1-yt}+\sum_{\substack{B\subseteq[n]\\1\in B,\, |B|\ge 2}}x^{|B|-1}y^{b_2-1-\eta(B)}\sum_{n\ge \max B}t^n\notag\\
    &=1+\frac{yt}{1-yt}+\frac{1}{1-t}\sum_{\substack{B\subseteq\mathbb{P}\\1\in B,\, |B|\ge 2}}x^{|B|-1}y^{b_2-1-\eta(B)}t^{\max B}.
\end{align}
Define the following enumerative polynomials
\begin{equation}\label{def-K}
    K(t;x,y)=\sum_{n\ge 2}\sum_{\substack{B\subseteq \mathbb{P}\\1\in B,\,|B|\ge 2}}x^{|B|-1}y^{b_2-1-\eta(B)}t^{\max B}.
\end{equation}
Combining this with \eqref{rewrite-P}, we have
\begin{equation}\label{gen-P-K}
    P_{\Pi}(t;x,y)=1+\frac{yt}{1-yt}+\frac{K(t;x,y)}{1-t}.
\end{equation}
 We now enumerate finite subsets $B$ of positive integers containing 1 and satisfying $|B|\ge 2$. Fix \(\ell:=b_2-1\geq 1\) and \(r\geq 1\). Then \(B\) begins
as
\[
B=\{1\}\cup\{\ell+1,\ell+2,\ldots,\ell+r\}\cup\{\text{later elements}\}.
\]
We now count the possible later elements. 
\begin{itemize}
    \item If there are no later elements, then
$\max B=\ell+r,$ and the contribution is \(t^{\ell+r}\). 
\item If there are later elements, then
\(\ell+r+1\notin B\), by maximality of the segment of consecutive elements, and
the next possible elements lie in
\[
\{\ell+r+2,\ell+r+3,\ldots \}.
\]
If the maximum is \(L:=\max B\geq \ell+r+2\), then \(L\) is forced to belong to \(B\),
and the elements of
\[
\{\ell+r+2,\ell+r+3,\ldots,L-1\}
\]
may be chosen freely. Thus the contribution of the later part is
\[
\sum_{L\geq \ell+r+2}
x(1+x)^{L-(\ell+r+2)}t^L
=
\frac{x t^{\ell+r+2}}{1-(1+x)t}.
\]
\end{itemize}
Therefore, for fixed \(\ell\) and \(r\), the contribution of the final part is
\begin{equation}\label{gen-fix-l-r}
    t^{\ell+r}
+
\frac{x t^{\ell+r+2}}{1-(1+x)t}
=
t^{\ell+r}
\frac{(1-t)(1-xt)}{1-(1+x)t}.
\end{equation}
Since \(x\) records \(|B|-1\), the segment $\ell+1,\ell+2,\ldots,\ell+r$ contributes \(x^r\). Combining this with \eqref{gen-fix-l-r}, we derive that $K(t;x,y)$ defined in \eqref{def-K} is equal to 
\begin{equation}\label{equ:K}
    K(t;x,y)=
\frac{(1-t)(1-xt)}{1-(1+x)t}
\sum_{\ell\geq 1}\sum_{r\geq 1}
x^r t^{\ell+r}
y^{\ell-\chi(2\,\nmid \,r)},
\end{equation}
where $\chi(P)$ is equal to $1$ if the proposition $P$ is true, and $0$ otherwise.
Separating \(r\) according to its parity, we get
\begin{align}
K(t;x,y)
&=
\frac{(1-t)(1-xt)}{1-(1+x)t}
\sum_{\ell\geq 1} t^\ell
\left(
y^\ell
\sum_{\substack{r\geq 1\\ r\text{ even}}}(xt)^r
+
y^{\ell-1}
\sum_{\substack{r\geq 1\\ r\text{ odd}}}(xt)^r
\right) \notag \\
&=
\frac{(1-t)(1-xt)}{1-(1+x)t}
\sum_{\ell\geq 1} t^\ell
\left(
y^\ell\frac{x^2t^2}{1-x^2t^2}
+
y^{\ell-1}\frac{xt}{1-x^2t^2}
\right) \notag\\
&=
\frac{x t^2(1+xyt)(1-t)(1-xt)}
{(1-yt)(1-x^2t^2)(1-(1+x)t)}.\label{gen-K}
\end{align}
Substituting $K(t;x,y)$ by \eqref{gen-K} in \eqref{gen-P-K}, we obtain \eqref{gen-P-132312}.
\end{proof}
\subsubsection{The class $\S_n(213,231)$} 
We first describe the structure of  permutations in $\S_n(213,231)$. Let $\sigma\in\S_n(213,231)$, and write 
\begin{equation}
  \sigma=L\,n\,R.  
\end{equation}
where $L$ and $R$ are the words to the left and right of the letter $n$, respectively. Since $\sigma$ avoids 213, the word $L$ must be increasing. Moreover, since $\sigma$ avoids $231$, every letter in $L$ must be smaller than every letter in $R$. It follows that
\begin{equation}
    L=12\cdots a\, \text{ for some } 0\le a\le n-1.
\end{equation}
More precisely,  every $\sigma\in\S_n(213,231)$ has a unique decomposition
\begin{equation}\label{equ:decom}
  \sigma=12\cdots a\,n\,\tau^{^{+}a}, 
\end{equation}
where $\tau^{^{+}a}$ is obtained by adding $a$ to each letter of $\tau\in\S_{n-a-1}(213,231)$.

\begin{example}
Let
$\sigma=1237465\in \mathfrak{S}_7(213,231).$
Then the letter \(7\) separates \(\sigma\) as
$\sigma=L\,7\,R$, where
$L=123$ and $R=\tau^{+3}=465.$ Thus \(a=3\). Subtracting \(a=3\) from each letter of \(R\), we obtain $\tau=132\in\S_3(213,231)$.
\end{example}

\begin{theorem}
    Let $T_2=\{213,231\}$. We have
 \begin{equation}\label{equ:T_2}
     F_{T_2}(t;x,y)=1+\frac{yt}{1-yt}
+\frac{x t^2(1+xyt)(1-xt)}{(1-yt)(1-x^2t^2)(1-(1+x)t)}.
 \end{equation}
\end{theorem}
\begin{proof}   
By Lemma \ref{re-com}, the reversal-complement map $\sigma\mapsto\sigma^{rc}$ restricts a bijection between $\S_n(132,312)$ and $\S_n(213,231)$, thus we have $F_{T_2}(t;x,y)=F_{T_1}(t;x,y), \text{ where }  T_1=\{132,312\}.$
    By the formula for $F_{T_1}(t;x,y)$ given in \eqref{gen-F-132312}, we obtain \eqref{equ:T_2}.
\end{proof}
\begin{theorem}
    Let $T_2=\{213,231\}$. We have
 \begin{equation}\label{equ:P_2}
     P_{T_2}(t;x,y)=1+\frac{yt}{1-yt}
+\frac{x t^2(1+xyt)(1-xt)}{(1-yt)(1-x^2t^2)(1-(1+x)t)}.
 \end{equation}
\end{theorem}
\begin{proof}
From the decomposition \eqref{equ:decom}, if $a=n-1$, then $\sigma=12\cdots n$. Hence, $\ides(\sigma)=0$ and $\pix(\sigma)=n$. Together with the empty permutation, these identity permutations contribute 
    \begin{equation}\label{equ:identity}
    1+\sum_{n\geq 1}y^n t^n=1+\frac{yt}{1-yt}.
 \end{equation}
 Now suppose $a\le n-2$, and put $m=n-a-1\ge 1$. Then
 \begin{equation}\label{dec:sigma-tau}
  \sigma=12\cdots a\,n\,\tau^{^{+}a}, \text{ with }\tau\in\S_{m}(213,231).
\end{equation}
Since $n$ occurs before all letters of $\tau^{^{+}a}$, the pair $(n-1,n)$ always creates one new inverse descent. The remaining inverse descents are exactly those coming from $\tau$. Hence,
\begin{equation}\label{ides-P-T_2}
    \ides(\sigma)=\ides(\tau)+1.
\end{equation}
For the pixed statistic, we distinguish whether $\tau$ is a desarrangement. If $\tau$ is a desarrangement, then the suffix $\tau^{^{+}a}$ is itself a desarrangement. By the pixed factorization of $\sigma$, we have $\pix(\sigma)=a+1$. 
If $\tau$ is not a desarrangement, then the prefix cannot include the letter $n$. In this case the suffix $n\,\tau^{^{+}a}$ is a desarrangement. Then we have $\pix(\sigma)=a$. Therefore,
 Therefore
\begin{equation}\label{pix-T_2}
    \pix(\sigma)=
\begin{cases}
a+1, & \text{if } \tau \text{ is a desarrangement};\\[5pt]
a, & \text{otherwise}.
\end{cases}
\end{equation}
Define the following two enumerative polynomials 
\begin{equation}
    A(t;x)=\sum_{n\ge 0}\Bigg(\sum_{\sigma\in\S_{n}(213,231)}x^{\ides(\sigma)}\Bigg)t^n,
\end{equation}
and 
\begin{equation}
    B(t;x)=\sum_{n\ge 1}\left(\sum_{\substack{\sigma\in\S_{n}(213,231)\\\sigma\text{ is a desarrangement}}}x^{\ides(\sigma)}\right)t^n.
\end{equation}

We first compute \(A(t;x)\). Together with the empty permutation, identity permutations contribute
\[
1+\sum_{n\ge1}t^n=\frac{1}{1-t}.
\]
For the non-identity part, using the decomposition \eqref{dec:sigma-tau}, the initial segment
\(12\cdots a\) contributes \(t^a\), the letter \(n\) contributes \(t\), and
by \eqref{ides-P-T_2} it also contributes a factor \(x\). The remaining nonempty
permutation \(\tau\) contributes \(A(t;x)-1\). Thus
\begin{equation}
    A(t;x)
=
\frac{1}{1-t}
+
\frac{xt}{1-t}\bigl(A(t;x)-1\bigr).
\end{equation}
Solving this equation gives
\begin{equation}\label{gen-A}
    A(t;x)=\frac{1-xt}{1-(1+x)t}.
\end{equation}

Next we compute \(B(t;x)\). A desarrangement in
\(\mathfrak S_n(213,231)\) cannot begin with a nonempty increasing segment
\(12\cdots a\) with \(a\ge1\), since then the first ascent occurs at position
\(1\). Hence, in the decomposition \eqref{equ:decom}, we must have \(a=0\). Thus
every nonempty desarrangement in this class is of the form
\[
\sigma=n\,\tau,\,\text{ with }
\tau\in\mathfrak S_{n-1}(213,231).
\]
Moreover, \(\tau\) is nonempty, since the singleton permutation is not a
desarrangement.

The first ascent of \(n\,\tau\) occurs one position later than the first ascent
of \(\tau\). Thus the parity of the position of the first ascent is reversed.
Consequently, \(n\,\tau\) is a desarrangement if and only if \(\tau\) is not a
desarrangement. Since the initial letter \(n\) creates one new inverse
descent, we obtain
\begin{equation}
    B(t;x)
=
xt\bigl(A(t;x)-1-B(t;x)\bigr).
\end{equation}
Therefore
\begin{equation}\label{gen-B}
    B(t;x)
=
\frac{xt(A(t;x)-1)}{1+xt}.
\end{equation}
Substituting $A(t;x)$ by \eqref{gen-A}, we obtain
\begin{equation}
    B(t;x)=
\frac{xt^2}
{(1+xt)(1-(1+x)t)}.
\end{equation}

We now return to \(P_{T_2}(t;x,y)\). The   contribution of empty permutation and  identity permutations is already
given by \eqref{equ:identity}. For the non-identity part, from the decomposition $\sigma=12\cdots a\,n\,\tau^{^{+}a}, \text{ with }\tau\in\S_{m}(213,231)$,  we see that the initial segment \(12\cdots a\) contributes \((yt)^a\), and summing over all
\(a\ge0\) gives
\begin{equation}
    \sum_{a\ge0}(yt)^a=\frac{1}{1-yt}.
\end{equation}
The letter \(n\) contributes \(t\) and creates one new inverse descent, hence
gives a factor \(xt\). Thus the common contribution is
\begin{equation}
    \frac{xt}{1-yt}.
\end{equation}

It remains to account for the contribution of \(\tau\). All nonempty
\(\tau\)'s contribute \(A(t;x)-1\). However, by \eqref{pix-T_2}, when \(\tau\) is
a desarrangement, the value of \(\operatorname{pix}(\sigma)\) is \(a+1\)
rather than \(a\), so these terms acquire one additional factor \(y\).
Therefore the contribution of the \(\tau\)-part is
\begin{equation}
    A(t;x)-1+(y-1)B(t;x).
\end{equation}
From the above argument, we see that the contribution of non-identity permutations is
\begin{equation}
    \frac{xt}{1-yt}
\left(
A(t;x)-1+(y-1)B(t;x)
\right).
\end{equation}
It follows that 
\begin{equation}
    P_{T_2}(t;x,y)=
1+\frac{yt}{1-yt}
+
\frac{xt}{1-yt}
\left(
A(t;x)-1+(y-1)B(t;x)
\right). 
\end{equation}
Substituting $A(t;x)$ and $B(t;x)$ by \eqref{gen-A} and \eqref{gen-B}, respectively,
we obtain \eqref{equ:P_2}.
\end{proof}

\subsubsection{The class $\S_n(132,321)$}

By the Simion--Schmidt structural description~\cite[Proposition~11]{SS85}, for some $1\le k\le m\le n$, every permutation $\sigma_{k,m}$ in $\S_n(132,321)$ has the form
\begin{equation}\label{form-sigma_n,k}
    \sigma_{k,m}:=(m-k+1)(m-k+2)\cdots m\,12\cdots(m-k)(m+1)\cdots n.
\end{equation}
Note that in the case $m=k=n$, $\sigma_{n,n}$ is the identity permutation.
\begin{theorem}\label{the:132321-F}
    Let $T_3=\{132,321\}$. We have
    \begin{equation}\label{equ:132321-F}
        F_{T_3}(t;x,y)=\frac{1}{1-yt}+\frac{x t^2}{(1-t)^2(1-yt)}.
    \end{equation}
\end{theorem}
\begin{proof}
    The identity permutation $\sigma_{n,n}$ contributes $y^n$ to the coefficient of $t^n$ in $F_{T_3}(t;x,y)$. For a non-identity element $\sigma_{k,m}$,
    by \eqref{form-sigma_n,k}, we see that
    the three displayed segments are increasing, i.e., $(m-k+1,\ldots,m)$, $(1,\ldots,m-k)$ and $(m+1,\ldots,n)$, and the only descent occurs between the first and second segments. Thus, we have
\begin{equation}
    \des(\sigma_{k,m})=1.
\end{equation}
The fixed points are precisely those in the final segment $m+1,m+2,\ldots,n$, and hence
\begin{equation}
    \fix(\sigma_{k,m})=n-m.
\end{equation}
For a fixed $m$, there are exactly $m-1$ non-identity permutations, and each permutation contributes $xy^{n-m}$. Thus,
\begin{equation}\label{rec-des-fix-F}
\sum_{\sigma\in\S_n(132,321)}x^{\des(\sigma)}y^{\fix(\sigma)}
=y^n+x\sum_{m=2}^n(m-1)y^{n-m}.
\end{equation}
Together with the contribution \(1\) of the empty permutation, multiplying \eqref{rec-des-fix-F} by $t^n$ and summing over $n\geq 1$ yields \eqref{equ:132321-F}.
\end{proof}
\begin{theorem}\label{the:132321-P}
Let $T_3=\{132,321\}$. We have
\begin{equation}\label{equ:132321-P}
    P_{T_3}(t;x,y)=\frac{1}{1-yt}+\frac{x t^2}{(1-t)^2(1-yt)}.
\end{equation}
\end{theorem}
\begin{proof}
For the identity permutation $\sigma_{n,n}$,  we have $\ides(\sigma_{n,n})=0$. By \eqref{form-sigma_n,k}, it is easy to see that, for $1\le k< m\le n$
\begin{equation}\label{form-inverse}
    \sigma_{k,m}^{-1}
    = (k+1)(k+2)\cdots m\,12\cdots k\,(m+1)\cdots n
    = \sigma_{m-k,m}.
\end{equation}
It follows that $\ides(\sigma_{k,m})=1$.  The length of the maximal consecutive increasing prefix of \(\sigma_{k,m}\) is equal to \(k\). In order to make the suffix of \(\sigma_{k,m}\) a desarrangement, we need to have the following form
\begin{equation}
    m,1,2,\dots,m-k,m+1,\dots,n.
\end{equation}
Hence,  we have $\pix(\sigma_{k,m})=k-1$. It follows that 
\begin{equation}\label{rec-des-fix-P}
\sum_{\sigma\in\S_n(132,321)}x^{\ides(\sigma)}y^{\pix(\sigma)}
=y^n+\sum_{m=2}^n\sum_{k=1}^{m-1}xy^{k-1}
=y^n+\sum_{j=0}^{n-2}(n-1-j)y^{j}.
\end{equation}
Together with the contribution \(1\) of the empty permutation, multiplying \eqref{rec-des-fix-P} by $t^n$ and summing over $n\geq 1$ yields \eqref{equ:132321-P}.
\end{proof}

\subsection{Triple-avoidance classes in $\mathcal{P}$}\label{sec:3}
In this section, we prove Theorem~\ref{thm:main} for the remaining  three-avoiding classes $\Pi$, namely $\{123,132,312\}$, $\{123,213,231\}$, $\{132,312,321\}$, $\{213,231,321\}$, $\{213,231,312\}$, and $\{123,312,321\}$ by computing their generating functions $F_{\Pi}(t;x,y)$ and $P_{\Pi}(t;x,y)$, respectively.
\subsubsection{The class \texorpdfstring{$\S_n(123,132,312)$}{Sn(123,132,312)}}
Following the Simion-Schmidt description~\cite[Proposition~16]{SS85}, the permutations in $\S_n(123,132,312)$ have the form
\begin{equation}\label{form-trible}
    \hat{\sigma}_{n,k}=(n-1)(n-2)\cdots k\,n\,(k-1)(k-2)\cdots 1, 
    \,\text{ for } 1\le k\le n.
\end{equation}
\begin{theorem}\label{the:123132312-F}
Let $T_4=\{123,132,312\}$. We have 
\begin{align}\label{equ:123132312-F}
F_{T_4}(t;x,y)=1+yt
+t^2\left(\frac{y^2+x}{1-x^2t^2}
+\frac{(1+y)x^2t^2}{(1-x^2t^2)^2}\right)
+t\left(\frac{(1+y)x t^2}{(1-x^2t^2)^2}
+\frac{y x^2t^2}{1-x^2t^2}\right).
\end{align}
\end{theorem}
\begin{proof}
By \eqref{form-trible}, we see that
\begin{equation}
    \sum_{\sigma\in\S_n(T_4)}
x^{\des(\sigma)}
y^{\fix(\sigma)}
=
\sum_{k=1}^{n}
x^{\des(\hat{\sigma}_{n,k})}
y^{\fix(\hat{\sigma}_{n,k})}.
\end{equation}

Let us first determine the number of descents of $\hat{\sigma}_{n,k}$. If $k<n$, then the only ascent occurs at position $n-k$, while all other positions $i\in[n-1]\setminus\{n-k\}$ are descents. If $k=n$, then 
$\hat{\sigma}_{n,n}=n(n-1)\cdots1$, so every position $i\in[n-1]$ is a descent. Therefore,
\begin{equation}\label{case-des}
    \des(\hat{\sigma}_{n,k})=
\begin{cases}
n-2, & \text{if } 1\le k<n;\\[4pt]
n-1, & \text{if } k=n.
\end{cases}
\end{equation}
We next compute the number of fixed points of $\hat{\sigma}_{n,k}$. From~\eqref{form-trible}, the letter $n$ occurs in position $n-k+1$. Hence, it is easy to see that the letter $n$ is a fixed point if and only if $k=1$.
 For $k\le a\le n-1$, the letter $a$ occurs in position $n-a$, and it is a fixed point if and only if $n$ is even and  $a=n/2$. 
 For $1\le a\le k-1$, the letter $a$ occurs in position $n-a+1$, and it is a fixed point if and only if $n$ is odd and  $a=(n+1)/2$. It follows that 
 \begin{equation}\label{fix-chi}
\fix(\hat{\sigma}_{n,k})=\chi(k=1)+\chi(n \text{ is even and } k\le n/2)+\chi(n \text{ is odd and } k> (n+1)/2).
 \end{equation}
Equivalently, if $n=2m+1$, then
\begin{equation}\label{case-fix-2m+1}
\fix(\hat{\sigma}_{2m+1,k})=
\begin{cases}
1, & k=1 \text{ or } m+2\le k\le 2m+1;\\[4pt]
0, & 2\le k\le m+1,
\end{cases}
\end{equation}
while if $n=2m$, then
\begin{equation}\label{case-fix-2m}
\fix(\hat{\sigma}_{2m,k})=
\begin{cases}
2, & k=1;\\[3pt]
1, & 2\le k\le m;\\[3pt]
0, & m+1\le k\le 2m.
\end{cases}
\end{equation}
Hence, for $n=2m+1$, by \eqref{case-fix-2m+1}, we have
\begin{equation}\label{odd-F-rec}
\sum_{\sigma\in\S_{2m+1}(T_4)}
x^{\des(\sigma)}y^{\fix(\sigma)}
=
x^{2m}y+x^{2m-1}m(1+y).
\end{equation}
For $n=2m$, by \eqref{case-fix-2m}, we have
\begin{equation}\label{even-F-rec}
\sum_{\sigma\in\S_{2m}(T_4)}
x^{\des(\sigma)}y^{\fix(\sigma)}
=
x^{2m-1}+x^{2m-2}\bigl((m-1)(1+y)+y^2\bigr).
\end{equation}
Together with the initial contributions \(1\) and \(yt\) corresponding to
\(n=0\) and \(n=1\), respectively, the stated generating function
\eqref{equ:123132312-F} follows by multiplying \eqref{odd-F-rec} and \eqref{even-F-rec} by \(t^{2m+1}\) and \(t^{2m}\), respectively, then summing
over \(m\ge 1\).
\end{proof}

\begin{theorem}\label{the:123132312}
Let $T_4=\{123,132,312\}$. We have 
\begin{equation}\label{equ:123132312-P}
P_{T_4}(t;x,y)=1+yt
+t^2\left(\frac{y^2+x}{1-x^2t^2}
+\frac{(1+y)x^2t^2}{(1-x^2t^2)^2}\right)
+t\left(\frac{(1+y)x t^2}{(1-x^2t^2)^2}
+\frac{y x^2t^2}{1-x^2t^2}\right).
\end{equation}
\end{theorem}

\begin{proof}
 First, let us determine $\ides$. From \eqref{form-trible}, a direct check gives
\begin{equation}\label{ides-hat}
\ides(\hat{\sigma}_{n,k})=
\begin{cases}
n-2, & \text{if } 1\le k<n;\\[4pt]
n-1, & \text{if } k=n.
\end{cases}
\end{equation}
Now we compute $\pix$. Recall that $\pix(\sigma)$ is the length of the initial increasing word in the pixed
factorization of $\sigma$. If \(k<n\), then the unique ascent of
\(\hat{\sigma}_{n,k}\) occurs at position \(n-k\). Therefore
\(\hat{\sigma}_{n,k}\) itself is a desarrangement precisely when
\(n-k\) is even. 

Suppose first that \(n=2m+1\). Then \(n-k\) is even exactly when \(k\)
is odd. Thus, for odd \(k<2m+1\), we have
\(\pix(\hat{\sigma}_{2m+1,k})=0\). For even \(k\), deleting the first
letter shifts the unique ascent to an even position, so the remaining
word is a desarrangement and \(\pix(\hat{\sigma}_{2m+1,k})=1\). Finally,
when \(k=2m+1\), deleting the first letter leaves a decreasing word of
even length, which is a desarrangement. Hence
\begin{equation}\label{odd-P-case}
\pix(\hat{\sigma}_{2m+1,k})=
\begin{cases}
0, & k \text{ is odd and } k<2m+1;\\[3pt]
1, & k \text{ is even};\\[3pt]
1,& k=2m+1.
\end{cases}
\end{equation}

Now suppose that \(n=2m\). Then \(n-k\) is even exactly when \(k\) is
even, so \(\pix(\hat{\sigma}_{2m,k})=0\) for even \(k\). If \(k\) is odd
and \(k<2m-1\), deleting the first letter shifts the unique ascent to an
even position, and hence \(\pix(\hat{\sigma}_{2m,k})=1\). In the remaining
case \(k=2m-1\), the first two letters form the increasing word
\((2m-1,2m)\), and deleting them leaves a decreasing word of even length,
which is a desarrangement.
If $n=2m$, then
\begin{equation}\label{even-P-case}
\pix(\hat{\sigma}_{2m,k})=
\begin{cases}
0, & k \text{ is even;}\\[3pt]
1, & k \text{ is odd and } k<2m-1;\\[3pt]
2, & k=2m-1.
\end{cases}
\end{equation}
Therefore, using \eqref{ides-hat} and \eqref{odd-P-case}, for $n=2m+1$,
\begin{equation}\label{odd-P-rec}
\sum_{\sigma\in\S_{2m+1}(T_4)}
x^{\ides(\sigma)}y^{\pix(\sigma)}
=
x^{2m}y+x^{2m-1}m(1+y),
\end{equation}
and using \eqref{ides-hat} and \eqref{even-P-case}, for $n=2m$,
\begin{equation}\label{even-P-rec}
\sum_{\sigma\in\S_{2m}(T_4)}
x^{\ides(\sigma)}y^{\pix(\sigma)}
=
x^{2m-1}+x^{2m-2}\bigl((m-1)(1+y)+y^2\bigr).
\end{equation}
Together with the initial contributions \(1\) and \(yt\) corresponding to
\(n=0\) and \(n=1\), respectively, the stated generating function
\eqref{equ:123132312-P} follows by multiplying \eqref{odd-P-rec} and \eqref{even-P-rec}  by \(t^{2m+1}\) and \(t^{2m}\), respectively,  then summing over \(m\ge 1\).
\end{proof}

\subsubsection{The class $\S_n(123,213,231)$}
Since \(\rho\) is a bijection between 
$\mathfrak S_n(123,132,312)$ and $\mathfrak S_n(123,213,231)$,  applying $\rho$ to \eqref{form-trible} shows that every permutation in
\(\mathfrak S_n(123,213,231)\) has the form
\begin{equation}\label{form-trible-2}
\tilde\sigma_{n,k}
=
n(n-1)\cdots(n-k+2)\,1\,(n-k+1)(n-k)\cdots2,\,\text{ for }
\quad 1\le k\le n.
\end{equation}
\begin{theorem}
    Let $T_5=\{123,213,231\}$. We have
    \begin{equation}
        F_{T_5}(t;x,y)=1+yt
+t^2\left(\frac{y^2+x}{1-x^2t^2}
+\frac{(1+y)x^2t^2}{(1-x^2t^2)^2}\right)
+t\left(\frac{(1+y)x t^2}{(1-x^2t^2)^2}
+\frac{y x^2t^2}{1-x^2t^2}\right).
    \end{equation}
\end{theorem}
\begin{proof}
By Lemma~\ref{re-com}, the reversal-complement map $\sigma\mapsto\sigma^{rc}$ restricts a bijection between $\S_n(123,213,231)$ and $\S_n(123,132,312)$, thus we have $F_{T_5}(t;x,y)=F_{T_4}(t;x,y)$, where $T_4=\{123,132,312\}$. Then
the stated formula follows immediately from Theorem~\ref{the:123132312-F}.
\end{proof}

\begin{theorem}
     Let $T_5=\{123,213,231\}$. We have
    \begin{align}
        P_{T_5}(t;x,y)=1+yt
+t^2\left(\frac{y^2+x}{1-x^2t^2}
+\frac{(1+y)x^2t^2}{(1-x^2t^2)^2}\right)
+t\left(\frac{(1+y)x t^2}{(1-x^2t^2)^2}
+\frac{y x^2t^2}{1-x^2t^2}\right).
    \end{align}
\end{theorem}

\begin{proof}

It is straightforward to see that 
\begin{equation}\label{case-ides-P}
    \ides(\tilde\sigma_{n,k})=
\begin{cases}
n-2, & 1\le k<n;\\[4pt]
n-1, & k=n.
\end{cases}
\end{equation}

Next we compute \(\pix(\tilde\sigma_{n,k})\). From the explicit form
of \(\tilde\sigma_{n,k}\), if \(k<n\), the unique ascent occurs at position
\(k\), namely the adjacent pair $1<n-k+1$.
Thus \(\tilde\sigma_{n,k}\) is itself a desarrangement exactly when \(k\)
is even. If \(k\) is odd and \(k<n\), deleting the first letter shifts this
unique ascent to position \(k-1\), which is even.

Now suppose \(n=2m+1\). If \(k\) is even, then
\(\tilde\sigma_{2m+1,k}\) is already a desarrangement, so its pix is \(0\).
If \(k\) is odd and \(k<2m+1\), deleting the first letter gives a
desarrangement, so the pix is \(1\). Finally, when \(k=2m+1\), the word
\(\tilde\sigma_{2m+1,2m+1}=(2m+1)(2m)\cdots1\) is decreasing of odd length;
deleting the first letter leaves a decreasing word of even length, which is
a desarrangement. Hence
\begin{equation}\label{case-pix-odd}
    \pix(\tilde\sigma_{2m+1,k})=
\begin{cases}
0, & k \text{ is even};\\[4pt]
1, & k \text{ is odd}.
\end{cases}
\end{equation}

Next suppose \(n=2m\). If \(k\) is even, then
\(\tilde\sigma_{2m,k}\) is already a desarrangement, so its pix is \(0\).
If \(k\) is odd and \(3\le k\le 2m-1\), deleting the first letter shifts
the unique ascent to the even position \(k-1\), so the pix is \(1\). The
remaining case is \(k=1\). In this case
$\tilde\sigma_{2m,1}=1~ 2m~ 2m-1 ~\ldots~2$
and the initial increasing word has length \(2\), since deleting the first
two letters leaves the decreasing word \((2m-1)\cdots2\) of even length,
which is a desarrangement. Therefore
\begin{equation}\label{case-pix-even}
    \pix(\tilde\sigma_{2m,k})=
\begin{cases}
0, & k \text{ is even};\\[3pt]
1, & k \text{ is odd and } 3\le k\le 2m-1;\\[3pt]
2, & k=1.
\end{cases}
\end{equation}

Combining \eqref{case-ides-P} and \eqref{case-pix-odd}, for \(n=2m+1\), we get
\begin{equation}\label{odd-t_5}
    \sum_{\sigma\in\mathfrak S_{2m+1}(T_5)}
x^{\ides(\sigma)}y^{\pix(\sigma)}
=
x^{2m}y+mx^{2m-1}(1+y).
\end{equation}
Similarly, combining \eqref{case-ides-P} and \eqref{case-pix-even}, for \(n=2m\), we get
\begin{equation}\label{even-t_5}
    \sum_{\sigma\in\mathfrak S_{2m}(T_5)}
x^{\ides(\sigma)}y^{\pix(\sigma)}
=
x^{2m-1}
+x^{2m-2}\bigl((m-1)(1+y)+y^2\bigr).
\end{equation}

Together with the initial contributions \(1\) and \(yt\) corresponding to
\(n=0\) and \(n=1\), respectively, the stated generating function follows
by multiplying \eqref{odd-t_5} and \eqref{even-t_5} by \(t^{2m+1}\) and \(t^{2m}\),
respectively, and summing over \(m\ge1\).
\end{proof}

\subsubsection{The class $\S_n(132,312,321)$} We claim that every permutation in
$\S_n(132,312,321)$ has the following form
\begin{equation}\label{form-sigma-d}
    \bar{\sigma}_{n,k}=2\,3\,\cdots k\,1\,(k+1)\,(k+2)\ldots n,\,
    \text{ for } 1\leq k\leq n.
\end{equation}
Here, the case \(k=1\) gives the identity permutation \(\bar{\sigma}_{n,1}=12\cdots n\).

Indeed, let \(\sigma\in\S_n(132,312,321)\), and let \(j=\sigma^{-1}(1)\) be the
position of \(1\). 
Since $\sigma$ avoids $321$,there are no descents among the letters lying to the left of $1$. Hence, the subword to the left of $1$ is increasing. Similarly, since \(\sigma\) avoids \(132\), the subword to the right of $1$ is also increasing. 

It remains to determine which letters occur on each side of $1$. Suppose that there  exist letters \(a\) to the left of \(1\) and \(b\) to the right of \(1\) with \(a>b\). Then the three letters \(a\,1\,b\), in this order, forms a \(312\)-pattern, contradicting the avoiding condition. Therefore, every letter to the left of \(1\) is smaller than every letter
to the right of \(1\). It follows that the letters to the left of \(1\)
are precisely \(2,3,\ldots,k\) for some \(k\), and the letters to the
right of \(1\) are \(k+1,k+2,\ldots,n\), that is \eqref{form-sigma-d}.

\begin{theorem}
    Let $T_6=\{132,312,321\}$. We have
\begin{equation}\label{equ:F-T_6}
    F_{T_6}(t;x,y)=1+\frac{yt}{1-yt}
+
\frac{x t^2}{(1-t)(1-yt)}.
\end{equation}
\end{theorem}

\begin{proof}
From the characterization \eqref{form-sigma-d}, for \(k=1\), we have \(\bar{\sigma}_{n,1}=12\cdots n\). Hence
\begin{equation}\label{des-fix-d}
    \des(\bar{\sigma}_{n,1})=0
\text{ and }
\fix(\bar{\sigma}_{n,1})=n.
\end{equation}
For \(2\leq k\leq n\), it is straightforward to see that the permutation \(\bar{\sigma}_{n,k}\) has exactly one descent, and its fixed points are exactly $k+1,k+2,\ldots,n$, i.e., 
\begin{equation}\label{des-fix-d-gen}
\des(\bar{\sigma}_{n,k})=1\text{ and }\fix(\bar{\sigma}_{n,k})=n-k.
\end{equation}
Therefore, combining \eqref{des-fix-d} and \eqref{des-fix-d-gen}, we obtain
\begin{equation}
    \sum_{\sigma\in\S_n(T_6)} x^{\des(\sigma)}y^{\fix(\sigma)}
=
y^n+\sum_{k=2}^n x y^{n-k}.
\end{equation}
Consequently, after multiplying by \(t^n\) and summing over \(n\geq 1\),
and after adding the initial contribution \(1\) of the empty permutation
corresponding to \(n=0\), we obtain \eqref{equ:F-T_6}.
\end{proof}

\begin{theorem}
    Let $T_6=\{132,312,321\}$. We have
\begin{equation}\label{equ:P-T_6}
    P_{T_6}(t;x,y)=1+\frac{yt}{1-yt}
+
\frac{x t^2}{(1-t)(1-yt)}.
\end{equation}
\end{theorem}

\begin{proof}
 From the characterization \eqref{form-sigma-d}, for \(k=1\), we have \(\bar{\sigma}_{n,1}=12\cdots n\). Hence
\begin{equation}\label{des-fix-T_6}
    \ides(\bar{\sigma}_{n,1})=0
\text{ and }
\pix(\bar{\sigma}_{n,1})=n.
\end{equation}
It remains to determine \(\pix(\bar{\sigma}_{n,r})\). Recall that \(\pix(\bar{\sigma})\)
is the length of the increasing prefix in the pixed factorization of
\(\bar{\sigma}\). For \(2\leq k\leq n\), the maximal increasing prefix of
\(\bar{\sigma}_{n,k}\) which occurs before the desarrangement is $23\cdots (k-1)$, that is,
\[
\bar{\sigma}_{n,k}
=
\underbrace{2\cdots(k-1)}_{\iota}\,
\underbrace{k1(k+1)\cdots(n-1)}_{\delta}.
\]
Thus, we have
\begin{equation}\label{pix-T_6}
    \pix(\bar{\sigma}_{n,k})=k-2.
\end{equation}
Therefore, combining \eqref{des-fix-T_6} and \eqref{pix-T_6},  we have
\begin{equation}
    \sum_{\sigma\in\S_n(T_6)} x^{\ides(\sigma)}y^{\pix(\sigma)}
=
y^n+\sum_{k=2}^n x y^{n-k}.
\end{equation}
Consequently, after multiplying by \(t^n\) and summing over \(n\geq 1\),
and after adding the initial contribution \(1\) of the empty permutation
corresponding to \(n=0\), we obtain \eqref{equ:P-T_6}.
\end{proof}

\subsubsection{The class $\S_n(213,231,321)$}
Since $\rho$ is a bijection between $\S_n(132,312,321)$ and $\S_n(213,231,321)$, by the form of permutations in $\S_n(132,312,321)$ (see \eqref{form-sigma-d}), we derive that every permutation in $\S_n(213,231,321)$ has the following form 
\begin{equation}\label{form-T_7}
    ^\star{\sigma}_{n,k}:=
1\,2\,\cdots\,(k-1)\,n\,k\,(k+1)\,\cdots\,(n-1),\,\text{ for } 1\le k\le n.
\end{equation}

\begin{theorem}
    Let $T_7=\{213,231,321\}$. We have
    \begin{equation}\label{equ:F-T_7}
        F_{T_7}(t;x,y)=1+\frac{yt}{1-yt}+\frac{x t^2}{(1-t)(1-yt)}.
    \end{equation}
\end{theorem}
\begin{proof}
    By Lemma~\ref{re-com}, the reversal-complement map $\sigma\mapsto\sigma^{rc}$ restricts a bijection between $\S_n(213,231,321)$ and $\S_n(132,312,321)$, thus we have $F_{T_7}(t;x,y)=F_{T_6}(t;x,y)$, where $T_6=\{132,312,321\}$. By the formula for $F_{T_6}(t;x,y)$ given in \eqref{equ:F-T_6},  we obtain \eqref{equ:F-T_7}.
\end{proof}

\begin{theorem}
    Let $T_7=\{213,231,321\}$. We have
    \begin{equation}\label{equ:P-T_7}
        P_{T_7}(t;x,y)=1+\frac{yt}{1-yt}+\frac{x t^2}{(1-t)(1-yt)}.
    \end{equation}
\end{theorem}
\begin{proof}
 From the characterization \eqref{form-T_7}, for \(k=n\), we have \(^\star{\sigma}_{n,n}=12\cdots n\). Hence,
\begin{equation}\label{T_7-special-case}
    \ides(^\star{\sigma}_{n,n})=0 
\text{ and }
\pix(^\star{\sigma}_{n,n})=n.
\end{equation}
For \(1\leq k\leq n-1\), one has
\begin{equation}\label{ides-T_7}
    \ides(^\star{\sigma}_{n,k})=1
\end{equation}

Recall that \(\pix(\pi)\) is the length of the initial increasing factor
in the pixed factorization of \(\pi\). For \(1\le k\le n-1\), we may
write
\[
{}^\star\sigma_{n,k}
=
\underbrace{12\cdots(k-1)}_{\iota}\,
\underbrace{nk(k+1)\cdots(n-1)}_{\delta}.
\]
The word \(\iota=12\cdots(k-1)\) is increasing. Moreover, the suffix
\(\delta=nk(k+1)\cdots(n-1)\) is a desarrangement: if \(k\le n-2\), then
its first ascent occurs at the second position, while if \(k=n-1\), then
\(\delta=n(n-1)\) has no ascent and is a desarrangement by convention.
Thus the above decomposition is precisely the pixed factorization of
\({}^\star\sigma_{n,k}\). Consequently,
\begin{equation}\label{pix-T_7}
    \operatorname{pix}({}^\star\sigma_{n,k})=k-1.
\end{equation}

Combining \eqref{T_7-special-case}, \eqref{ides-T_7} and \eqref{pix-T_7}, we obtain, for \(n\ge1\),
\begin{equation}
    \sum_{\sigma\in\mathfrak{S}_n(T_7)}
x^{\operatorname{ides}(\sigma)}
y^{\operatorname{pix}(\sigma)}
=
y^n+\sum_{k=1}^{n-1}xy^{k-1}.
\end{equation}
Consequently, after multiplying by \(t^n\) and summing over \(n\geq 1\),
and after adding the initial contribution \(1\) of the empty permutation
corresponding to \(n=0\), we obtain \eqref{equ:P-T_7}.
\end{proof}

\subsubsection{The class $\S_n(213,231,312)$}
We claim that every permutation in \(\mathfrak{S}_n(213,231,312)\) is of the form
\begin{equation}\label{form-sigma_nk}
    \mathring{\sigma}_{n,k}=12\cdots k\, n\, (n-1)\cdots(k+1),\,\text{ for } 0\le k\le n-1.
\end{equation}

Indeed, let \(\sigma\in\mathfrak{S}_n(213,231,312)\), and consider the
position of the letter \(1\). Since \(1\) is the smallest letter, if there
were letters on both sides of \(1\), say \(a\) to its left and \(b\) to its
right, then the subsequence \(a 1 b\) would form a \(213\)-pattern if
\(a<b\), and a \(312\)-pattern if \(a>b\). Thus the letter \(1\) must occur
at one of the two ends of \(\sigma\).

If \(1\) is the last letter, then the subword preceding it must be decreasing.
Indeed, any increasing pair \(a<b\) occurring before \(1\) would give a
\(231\)-pattern \(a b 1\). Hence, in this case,
\[
\sigma=n(n-1)\cdots 21,
\]
which corresponds to \(k=0\) in \eqref{form-sigma_nk}.

Otherwise, \(1\) must be the first letter. Deleting this initial \(1\) and
standardizing the remaining word, we obtain a permutation of
\(\mathfrak{S}_{n-1}(213,231,312)\). Repeating the same argument, we find
that the letters \(1,2,\ldots,k\) must appear successively at the beginning
of \(\sigma\), until the first step at which the smallest remaining letter
is placed at the right end. At that point, the remaining letters must appear
in decreasing order. Consequently, for some \(0\le k\le n-1\),
\[
\sigma=\mathring{\sigma}_{n,k}=12\cdots k\, n\, (n-1)\cdots(k+1),
\]
as claimed.
\begin{theorem}
    Let $T_8=\S_n(213,231,312)$. We have
\begin{equation}\label{equ:T_8}
    F_{T_8}(t;x,y)=1+\frac{yt+xt^2}{(1-yt)(1-x^2t^2)}.
\end{equation}
\end{theorem}
\begin{proof}  
From \eqref{form-sigma_nk}, the initial segment \(12\cdots k\) is increasing,
whereas the final segment
$n(n-1)\cdots(k+1)$
is decreasing and has length \(n-k\). Hence the descents occur precisely
inside this final decreasing segment. Therefore
\begin{equation}\label{des-T_8}
\des(\mathring{\sigma}_{n,k})=n-k-1.
\end{equation}
It remains to determine the number of fixed points. The letters
\(1,2,\ldots,k\) are fixed points, and hence contribute \(k\) fixed points.
Now consider the final decreasing segment. Its positions are
$k+1,k+2,\ldots,n,$ and its corresponding values are
$n,n-1,\ldots,k+1.$
The \(j\)-th entry of this segment, where \(1\le j\le n-k\), lies in position
\(k+j\) and has value \(n-j+1\). It is fixed if and only if
$k+j=n-j+1,$ or equivalently, $2j=n-k+1.$
Thus the final decreasing segment contributes exactly one additional fixed
point when \(n-k\) is odd. Consequently,
\begin{equation}\label{fix-T_8}
    \operatorname{fix}(\mathring{\sigma}_{n,k})
=
k+\chi(n-k\ \text{is odd}).
\end{equation}

Combining \eqref{des-T_8} and \eqref{fix-T_8}, we obtain, for \(n\ge1\),
\begin{equation}\label{rec-T_8}
    \sum_{\sigma\in\mathfrak{S}_n(T_8)}
x^{\operatorname{des}(\sigma)}
y^{\operatorname{fix}(\sigma)}
=
\sum_{k=0}^{n-1}
x^{n-k-1}
y^{k+\chi(n-k\ \text{is odd})}.
\end{equation}
Taking also into account the empty permutation for \(n=0\), which contributes
\(1\), and multiplying both side of \eqref{rec-T_8} by  $t^n$ and summing over $n\ge 1$, we get
\[
\begin{aligned}
F_{T_8}(t;x,y)
&=
1+\sum_{n\ge1}t^n
\sum_{\sigma\in\mathfrak{S}_n(T_8)}
x^{\operatorname{des}(\sigma)}
y^{\operatorname{fix}(\sigma)}  \\
&=
1+\sum_{n\ge1}\sum_{k=0}^{n-1}
t^n x^{n-k-1}
y^{k+\chi(n-k\ \text{is odd})}.
\end{aligned}
\]
Setting \(\ell=n-k\), the length of the final decreasing segment in the above identity. Then for
\(\ell\ge1\), \(k\ge0\) and \(n=k+\ell\), we have 
\begin{align}\label{F-T_8}
F_{T_8}(t;x,y)
&=
1+\sum_{k\ge0}\sum_{\ell\ge1}
t^{k+\ell}x^{\ell-1}
y^{k+\chi(\ell\ \text{is odd})} \notag \\
&=
1+\left(\sum_{k\ge0}(yt)^k\right)
\left(\sum_{\ell\ge1}t^\ell x^{\ell-1}
y^{\chi(\ell\ \text{is odd})}\right) \notag \\
&=
1+\frac{1}{1-yt}
\sum_{\ell\ge1}t^\ell x^{\ell-1}
y^{\chi(\ell\ \text{is odd})}.
\end{align}
Splitting according to the parity of
\(\ell\), the last sum of \eqref{F-T_8} is equal to
\begin{align}
    \sum_{\ell\ge1}t^\ell x^{\ell-1}
y^{\chi(\ell\ \text{is odd})}
&=
\sum_{m\ge0}t^{2m+1}x^{2m}y
+
\sum_{m\ge1}t^{2m}x^{2m-1}
=\frac{yt+xt^2}{1-x^2t^2}.
\end{align}
Combining this and \eqref{F-T_8} yields \eqref{equ:T_8}.
\end{proof}

\begin{theorem}
    Let $T_8=\S_n(213,231,312)$. We have
\begin{equation}\label{equ:P-T_8}
    P_{T_8}(t;x,y)=1+\frac{yt+xt^2}{(1-yt)(1-x^2t^2)}.
\end{equation}
\end{theorem}
\begin{proof}
Recall that \(\operatorname{ides}(\sigma)\) counts the values
\(v\in[n-1]\) such that \(v+1\) appears before \(v\) in \(\sigma\). By \eqref{form-sigma_nk}, we see that each $v$ is an inverse descent of $\mathring{\sigma}_{n,k}$ for \(k+1\le v\le n-1\). Since the final segment is
decreasing, while the initial segment \(12\cdots k\) is increasing and \(k\)
appears before \(k+1\). Hence, $k$ is also an inverse descent of $\mathring{\sigma}_{n,k}$. 
\begin{equation}\label{ides-T_8}
    \ides(\mathring{\sigma}_{n,k})=n-k-1.
\end{equation}

It remains to determine the number of pixed points. For \(\mathring{\sigma}_{n,k}\), the initial increasing
part is $12\cdots k,$ and the remaining suffix is the decreasing word
$n(n-1)\cdots(k+1),$ of length \(n-k\). We distinguish two cases according to the parity of \(n-k\). If \(n-k\) is
even, then the decreasing suffix \(n(n-1)\cdots(k+1)\) is a desarrangement.
Thus the pixed factorization is
\[
\mathring{\sigma}_{n,k}
=
\underbrace{12\cdots k}_{\iota}
\,
\underbrace{n(n-1)\cdots(k+1)}_{\delta},
\]
and consequently $\pix(\mathring{\sigma}_{n,k})=k.$

If \(n-k\) is odd, then the decreasing suffix
\(n(n-1)\cdots(k+1)\) is not a desarrangement. In this case, the first
letter \(n\) of the final segment must be included in the increasing prefix.
The remaining suffix
$(n-1)(n-2)\cdots(k+1)$
has even length and is therefore a desarrangement. Hence the pixed
factorization is
\[
\mathring{\sigma}_{n,k}
=
\underbrace{12\cdots k\,n}_{\iota}
\,
\underbrace{(n-1)(n-2)\cdots(k+1)}_{\delta},
\]
and so
$\pix(\mathring{\sigma}_{n,k})=k+1$. Thus, in all cases, we have
\begin{equation}\label{pix-T_8}
    \pix(\mathring{\sigma}_{n,k})
=
k+\chi(n-k\ \text{is odd}).
\end{equation}
Combining \eqref{ides-T_8} and \eqref{pix-T_8}, 
we obtain, for fixed \(n\ge1\),
\begin{equation}\label{rec-P-T_8}
    \sum_{\sigma\in\mathfrak{S}_n(T_8)}
x^{\operatorname{ides}(\sigma)}
y^{\operatorname{pix}(\sigma)}
=
\sum_{k=0}^{n-1}
x^{n-k-1}
y^{k+\chi(n-k\ \text{is odd})}.
\end{equation}
Since the right-hand side of \eqref{rec-P-T_8} is identical to that of \eqref{rec-T_8}, the same summation argument as above gives $F_{T_8}(t;x,y)=P_{T_8}(t;x,y)$.
Hence, by \eqref{equ:T_8}, we obtain the desired formula.  
\end{proof}

\subsubsection{The class $\S_n(123,312,321)$}
Let \(T_9=\{123,312,321\}\). By the Erdős--Szekeres theorem~\cite[p.160]{Coh78},
\(\mathfrak S_n(123,321)=\varnothing\) for \(n\ge5\), and therefore
$\mathfrak S_n(T_9)=\varnothing \text{ for } n\ge5.$
For \(1\le n\le4\), direct enumeration yields
$\mathfrak S_1(T_9)=\{1\}$,
$\mathfrak S_2(T_9)=\{12,21\},$
$\mathfrak S_3(T_9)=\{132,213,231\}$, and
$\mathfrak S_4(T_9)=\{2143\}.$
Consequently, we have
\[
F_{T_9}(t;x,y)=P_{T_9}(t;p,q)
=1+yt+(y^2+x)t^2+x(2y+1)t^3+x^2t^4 .
\]

\section{Concluding remarks}

\subsection*{Acknowledgement}
We thank Zhicong Lin for helpful discussions related to this work. The second author was supported by the China Scholarship Council (No. 202206220034).
\printbibliography
\end{document}